\documentclass[11pt,reqno]{amsart}
\usepackage{amssymb,latexsym}
\usepackage{graphicx}
\usepackage{mathtools}
\usepackage{color}
\usepackage[hyphens]{url}
\usepackage[T1]{fontenc}
\usepackage{hyperref}
\usepackage{amsmath}
\usepackage{mathdots}
\usepackage{breqn}
\usepackage[toc,page]{appendix}
\usepackage{url}    
\usepackage{breqn}
\usepackage{hyperref}
\usepackage{breakurl}
\newcommand{\bburl}[1]{\textcolor{blue}{\url{#1}}}

\calclayout

\makeatletter
\newcommand{\monthyear}[1]{%
  \def\@monthyear{\uppercase{#1}}}
\newcommand{\volnumber}[1]{%
  \def\@volnumber{\uppercase{#1}}}
\AtBeginDocument{%
\def\ps@plain{\ps@empty
  \def\@oddfoot{\@monthyear \hfil \thepage}%
  \def\@evenfoot{\thepage \hfil \@volnumber}}
\def\ps@firstpage{\ps@plain}
\def\ps@headings{\ps@empty
  \def\@evenhead{%
    \setTrue{runhead}%
    \def\thanks{\protect\thanks@warning}%
    \uppercase{\ }\hfil}%
  \def\@oddhead{%
    \setTrue{runhead}%
    \def\thanks{\protect\thanks@warning}%
    \hfill\uppercase{Partial Sums of the Fibonacci Sequence}}%
  \let\@mkboth\markboth
  \def\@evenfoot{%
    \thepage \hfil \@volnumber}%
  \def\@oddfoot{%
    \@monthyear \hfil \thepage}%
  }%
\footskip=25pt
\pagestyle{headings}%
}
\makeatother

\theoremstyle{plain}
\numberwithin{equation}{section}
\newtheorem{thm}{Theorem}[section]

\newcommand{\ignore}[1]{}

\newcommand\be{\begin{eqnarray}}
\newcommand\ee{\end{eqnarray}}
\newcommand\bea{\begin{eqnarray}}
\newcommand\eea{\end{eqnarray}}
\newcommand\ben{\begin{enumerate}}
\newcommand\een{\end{enumerate}}

\newtheorem{cor}[thm]{Corollary}
\newtheorem{lem}[thm]{Lemma}

\newtheorem{rek}[thm]{Remark}

\begin{document}

\monthyear{}
\volnumber{Volume, Number}
\setcounter{page}{1}
\title{Partial Sums of the Fibonacci Sequence}

\author{H\`ung Vi\d{\^e}t Chu}

\address{Department of Mathematics, University of Illinois at Urbana-Champaign, Urbana, IL 61820} \email{hungchu2@illinois.edu}

\date{\today}

\begin{abstract}
Let $(F_n)_{n\ge 1}$ be the Fibonacci sequence. Define $P(F_n): = (\sum_{i=1}^n F_i)_{n\ge 1}$; that is, the function $P$ gives the sequence of partial sums of $(F_n)$. In this paper, we first give an identity involving $P^k(F_n)$, which is the resulting sequence from applying $P$ to $(F_n)$ $k$ times. Second, we provide a combinatorial interpretation of the numbers in $P^k(F_n)$. 
\end{abstract}

\thanks{The author is thankful for the anonymous referee's comments that improved the clarity of this paper.}

\maketitle
\section{Introduction}
Let $(F_n)_{n\ge 1}$ be the Fibonacci sequence with $F_1 = 1, F_2 = 1$, and $F_{n+2} = F_{n+1} + F_{n}$ for $n\ge 1$. Define $P(F_n): = (\sum_{i=1}^n F_i)_{n\ge 1}$; that is, the function $P$ gives the sequence of partial sums of $(F_n)$. Generally, we can apply the function $P$ to $(F_n)$ $k$ times to have the sequence $P^k(F_n)$. For ease of notation, let $a_k(n)$ denote the $n$th number in the sequence $P^k(F_n)$. Table 1 summarizes some initial numbers $a_k(n)$ for different values of $k$ and $n$.\newline

\begin{tabular}{c|cccccccccccc}
$k\backslash n$ & $1$ & $2$ & $3$ & $4$ & $5$ & $6$ & $7$ & $8$ & $9$ & $10$ & $11$ & $12$\\
\hline
0 & 1 & 1 & 2 & 3        & $5$ &    $8$ &   $13$ & $21$ & $34$ & $55$ & $89$ & $144$\\
1 & 1 & 2 & $4$ & $7$    & $12$ &   $20$&   $33$ & $54$ & $88$ & $143$ & $232$ & $376$\\
2 & 1 & 3 & $7$ & $14$   & $26$ &   $46$&   $79$ & $133$ & $221$ & $364$ & $596$ & $972$\\
3 & 1 & 4 & $11$ & 25    & $51$ &   $97$&   $176$ & $309$ & $530$ & $894$ & $1490$ & $2462$\\
4 & 1 & 5 & $16$ & $41$  & $92$ &   $189$&  $365$ & $674$ & $1204$ & $2098$ & $3588$ & $6050$\\
5 & 1 & 6 & $22$ & $63$  & $155$ &  $344$&  $709$ & $1383$ & $2587$ & $4685$ & $8273$ & $14323$\\
\end{tabular}

\begin{center}
Table 1. Initial numbers $a_k(n)$ for different values of $k$ and $n$.
\end{center}
We compute a few values of $a_k(n)$ to illustrate how Table 1 is obtained. For $k = 0$, we do not apply the partial sum operator to the first row, so we have the Fibonacci sequence. For $k = 1$, we apply the partial sum operator once to the first row to have the second row. Thus, 
\begin{align*}
    &a_2(1) \ =\  a_1(1)\ =\ 1,\\
    &a_2(2) \ =\ a_1(1) + a_1(2) \ =\ 1 + 1 = 2,\\
    &a_2(3) \ =\  a_1(1) + a_1(2) + a_1(3) \ =\ 1 + 1 + 2 = 4,
\end{align*}
and so on. For $k = 2$, we apply the partial sum operator twice to the first row; equivalently, we apply the partial sum operator once to the second row. Repeat the procedure to fill in the table. 

We now consider the first and second rows. Observe that each number in the second row is 1 less than a Fibonacci number; in particular, $a_0(n+2) - 1 = a_1(n)$ for all $n\ge 1$. A natural question is about the relationship between any two consecutive rows of Table 1. Still looking at pairs of numbers whose columns differ by 2, we compute several differences 
\begin{align*}
    &a_1(11) - a_2(9) \ =\ 11,\\
    &a_2(10) - a_3(8)  \ =\ 55, \\
    &a_3(8) - a_4(6) \ =\ 309-189 \ =\ 120,\\ 
    &a_4(12) - a_5(10) \ =\ 6050-4685 \ =\ 1365.
\end{align*}
If we check these numbers in the Pascal triangle, we see a pattern. Call the outmost left edge of the Pascal triangle the \textit{first line} (consisting of $1$'s); move to the right 1 unit to have the \textit{second line} (consisting of $1, 2, 3, 4, \ldots$), and so on. Then $11$ lies on the second line; $55$ lies on the third line; $120$ lies on the fourth line; $1365$ lies on the fifth line. This fact might lead to the suspicion that $a_{k-1}(n+2) - a_{k}(n)$ lies on the $k$th line. The following theorem shows that our suspicion is indeed well-founded. 
\begin{thm}\label{m1}
For $k\ge 1$, we have
$$a_k(n)\ =\ a_{k-1}(n+2) - \binom{n+k}{k-1}.$$
By definition, the identity is equivalent to $$\sum_{m=1}^n a_{k-1}(m) \ =\ a_{k-1}(n+2) - \binom{n+k}{k-1}.$$
\end{thm}
\begin{rek}\normalfont
If we let $k = 1$, then Theorem \ref{m1} gives the identity $\sum_{m=1}^n a_0(m)  = a_{0}(n+2) - 1$, which is the well known identity 
$\sum_{m=1}^n F_m  = F_{n+2} - 1$. 
\end{rek}

Our next result gives a combinatorial interpretation of $a_k(n)$. For $n\in\mathbb{N}$, define
$$s_k(n) := \#\{S\subset \{1, 2, \ldots, n\}\,:\, |S| \ge k \mbox{ and } \min S\ge |S|\}.$$
Sets with $\min S\ge |S|$ are called Schreier sets. Schreier \cite{S} used these sets to solve a problem in Banach space theory and they were also independently discovered in combinatorics and are connected to Ramsey-type theorems for subsets of $\mathbb{N}$. Table 2 summarizes some initial numbers $s_k(n)$ for different values of $k$ and $n$.\newline 

\begin{tabular}{c|cccccccccccc}
$k\backslash n$ & $1$ & $2$ & $3$ & $4$ & $5$ & $6$ & $7$ & $8$ & $9$ & $10$ & $11$ & $12$\\
\hline
0 & 2 & 3 & 5 & 8 & 13        & $21$ &    $34$ &   $55$ & $89$ & $144$ & $233$ & $377$ \\
1 & 1 & 2 & $4$ & $7$    & $12$ &   $20$&   $33$ & $54$ & $88$ & $143$ & $232$ & $376$\\
2 & 0 & 0 & 1 & 3 & $7$ & $14$   & $26$ &   $46$&   $79$ & $133$ & $221$ & $364$ \\
3 & 0 & 0 & 0 & 0 & 1 & 4 & $11$ & 25    & $51$ &   $97$&   $176$ & $309$ \\
4 & 0 & 0 & 0 & 0 & 0 & 0 & 1 & 5 & $16$ & $41$  & $92$ &   $189$ \\
5 & 0 & 0 & 0 & 0 & 0 & 0 & 0 & 0 & 1 & 6 & $22$ & $63$  \\
\end{tabular}

\begin{center}
Table 2. Initial numbers $s_k(n)$ for different values of $k$ and $n$.
\end{center}
We obtain Table 2 with the help of a simple program. Comparing Tables 1 and 2, we see that $s_k(n)$ is a shift of $a_k(n)$. The following theorem describes this relationship. 
\begin{thm}\label{m2}
For $k\ge 0$, we have 
$$s_k(n) \ =\ a_k(n-2(k-1)).$$
\end{thm}
\begin{rek}
\normalfont
In Theorem \ref{m2}, if $n-(2k-1)\le 0$, then $a_k(n-2(k-1)) = 0$ by convention. 
\end{rek}
\section{Proofs}
We use induction repeatedly and several straightforward observations to prove Theorems \ref{m1} and \ref{m2}. The following lemma is useful.
\begin{lem}\label{fa3}
For $k\ge 0$, we have $a_k(3) = \binom{k+2}{k} + 1$. 
\end{lem}
\begin{proof}
We prove by induction. 

\textit{Base case:} for $k = 0$, $a_0(3) = F_3 = 2 = \binom{2}{0} + 1$. 

\textit{Inductive hypothesis:} suppose that the formula holds for all $k\le \ell$ for some $\ell\ge 0$. We show that it holds for $k = \ell+1$. We have 
\begin{align*}
    a_{\ell+1}(3) &\ =\ a_{\ell+1}(2) + a_{\ell}(3) & \mbox{ by how we construct Table 1 }\\
    &\ =\ (\ell+2) + \left(\binom{\ell+2}{\ell} + 1\right) & \mbox{ by the inductive hypothesis}\\
    &\ =\ \binom{\ell+2}{\ell+1} + \left(\binom{\ell+2}{\ell} + 1\right)\\
    &\ =\ \binom{\ell+3}{\ell+1} + 1.
\end{align*}
This completes our proof. 
\end{proof}

\begin{proof}[Proof of Theorem \ref{m1}]
We prove by induction on $k$.

\textit{Base case:} for $k = 1$, the identity is $\sum_{m=1}^n a_0(m) = a_0(n+2)-1$, which is true because $\sum_{m=1}^n F_m = F_{n+2} - 1$. 

\textit{Inductive hypothesis:} Suppose that the identity is true for $k \le \ell$ for some $\ell \ge 1$. We show that \begin{align} \label{k1} \sum_{m=1}^n a_{\ell}(m) \ =\ a_\ell(n+2) - \binom{n+\ell+1}{\ell}.\end{align}

We prove \eqref{k1} by induction on $n$. Base case: if $n = 1$, then the left side is $a_{\ell}(1) = 1$, while the right side is $a_\ell(3)- \binom{\ell+2}{\ell} = 1$ by Lemma \ref{fa3}. Inductive hypothesis: suppose that \eqref{k1} holds for $n\le j$ for some $j\ge 1$. We show that it holds for $n = j+1$; that is, $$\sum_{m=1}^{j+1}a_\ell(m) = a_\ell (j+3) - \binom{j+\ell+2}{\ell}.$$ 
We have
\begin{align*}
    \sum_{m=1}^{j+1}a_\ell(m) &\ =\ a_{\ell}(j+1) + \sum_{m=1}^{j} a_\ell(m)\\
    &\ =\ \left(a_{\ell-1}(j+3) - \binom{j+\ell+1}{\ell-1}\right) + \left(a_\ell(j+2) - \binom{j+\ell+1}{\ell}\right)\\
    & \mbox{ by our inductive hypotheses}\\
    &\ =\ (a_{\ell-1}(j+3)+a_\ell(j+2)) - \binom{j+\ell+2}{\ell}\\
    &\ =\ a_{\ell}(j+3) - \binom{j+\ell+2}{\ell}.
\end{align*}
This completes our proof. 
\end{proof}

Next, we prove Theorem \ref{m2}. We will need the following lemma.  

\begin{lem}
Fix $n\ge 1$. We have
$$\#\{S\subset \{1,2, \ldots, n\}\, :\, |S| = k \mbox{ and } \min S \ge k\} \ =\ \binom{n-k+1}{k}.$$
\end{lem}

\begin{proof}
Observe that to form a set $S\subset \{1, 2, \ldots, n\}$ with $|S| = k$ and $\min S\ge k$, we choose $k$ numbers in $\{1, 2, \ldots, n\}$ such that the smallest number is at least $k$. To satisfy these requirements, we need to choose $k$ numbers from $k$ to $n$ inclusive. 
Therefore, the number of ways is $\binom{n-k+1}{k}$. 
\end{proof}

The following corollary is immediate. 
\begin{cor}\label{cs}
Fix $\ell\ge 0$ and $n\ge 1$. We have
$$s_{\ell+1}(n)\ =\ s_{\ell}(n) - \binom{n-\ell+1}{\ell}.$$
\end{cor}

\begin{proof}[Proof of Theorem \ref{m2}] We prove by induction on $k$.

\textit{Base case:} for $k = 0$, we want to show that $s_0(n) = a_0(n+2): = F_{n+2}$. This follows from \cite[Theorem 1]{Ch}.

\textit{Inductive hypothesis:} Suppose that the formula holds for $k\le \ell$ for some $\ell\ge 0$. We want to show that 
\begin{align}\label{k2}
    s_{\ell+1}(n) \ =\ a_{\ell+1}(n-2\ell).
\end{align}

For $n\le 2\ell$, the right side of \eqref{k2} is $0$ by convention. The left side counts the number of subsets $S\subset \{1, 2, \ldots, n\}$ such that $|S|\ge \ell+1$ and $\min S\ge |S|\ge  \ell+1$. Because $|S|\ge \ell+1 > n/2$, we know that $S$ must contain a number smaller than or equal to $n/2\le \ell$. Hence, $\min S\le \ell$, which contradicts that $\min S\ge \ell+1$. Therefore, the left side is also $0$. 

For $n> 2\ell$, we have $$a_{\ell+1}(n-2\ell) = a_\ell(n-2\ell+2) - \binom{n-\ell+1}{\ell}$$ by Theorem \ref{m1}. On the other hand, we have
$$s_{\ell+1}(n) = s_{\ell}(n)-\binom{n-\ell+1}{\ell}$$ by Corollary \ref{cs}. By the inductive hypothesis, $s_{\ell}(n) = a_\ell(n-2\ell+2)$; therefore, we have
$s_{\ell+1}(n) \ =\ a_{\ell+1}(n-2\ell)$. This completes our proof. 
\end{proof}



\newcommand{\etalchar}[1]{$^{#1}$}

\ \\

\noindent MSC2010: 11B39

\end{document}